\documentclass[11pt]{article}
 
 \usepackage{amsmath,amsthm,verbatim,amssymb,amsfonts,amscd, graphicx}
 \usepackage{graphics}
 \usepackage{cite}
 \usepackage{epstopdf}
 \usepackage{float}
 \usepackage{hyperref,color}
 \usepackage{mathtools}
 \theoremstyle{plain}
 \newtheorem{theorem}{Theorem}
 
 \newtheorem{lemma}{Lemma}
 \newtheorem{proposition}{Proposition}

 \theoremstyle{definition}
 
 \DeclareMathOperator*{\argmin}{argmin}

\newcommand{\pa}{\Psi_{\alpha}}
\newcommand{\pp}[1]{\Psi_{#1}}
\newcommand{\epss}{\epsilon_2}

 \title{Heuristic Parameter Choice Rules for Tikhonov Regularisation with Weakly Bounded Noise}
 \author{Stefan Kindermann\thanks{Industrial Mathematics Institute, Johannes Kepler University Linz, Altenbergerstra{\ss}e 69, 4040, Linz, Austria. (kindermann@indmath.uni-linz.ac.at \& kemal.raik@indmath.uni-linz.ac.at). This work was supported by the Austrian Science
 Fund (FWF) project P 30157-N31. }, Kemal Raik\footnotemark[1]}
 
\begin{document}
 	\maketitle
 	\begin{abstract}
 		We study the choice of the regularisation parameter for linear ill-posed problems in the presence of noise that is possibly unbounded but only finite in a weaker norm, 
 		and when the noise-level is unknown. For this task, we analyse  several 
 		heuristic parameter choice rules, such as  the quasi-optimality, 
 		heuristic discrepancy, and Hanke-Raus rules and adapt the latter two to 
 		the weakly bounded noise case.  We prove convergence and convergence rates under 
 		certain noise conditions. Moreover, we analyse  and provide conditions for  the convergence of 
 		the parameter choice by the generalised cross-validation and predictive mean-square error rules. 
 	\end{abstract}
 	\section{Introduction}	
 	Let $X$ and $Y$ be Hilbert spaces and $T:X\to Y$ be a compact linear operator.
 	We consider the ill-posed problem
 	\begin{equation}
 		Tx=y, \label{illposedproblem}
 	\end{equation}
 	in which $T$ may have a nontrivial kernel and where we do not know $y$ exactly, but only noisy data $y^\delta=y+e$
 	are available. In contrast to the standard setting, 
 	the  main focus of this paper concerns the case of possibly unbounded noise, i.e., 
 	  $\delta:=\|e\|$ is possibly infinite. The latter may occur, for instance, in the case where we have \textit{white noise} and $Y$ is the space of square summable sequences. It may be, however, that the noise is \textit{weakly bounded} (cf.~\cite{weaklybounded,mathegeneralnoise,egger,morozov,kokurin}), which we define as being whenever 
 	  \begin{equation}\label{eq:etadef}
 	  \eta:=\|(TT^\ast)^p(y^\delta-y)\|< \infty, \qquad \text{ for some } \quad p\in[0,\frac{1}{2}].
 	  \end{equation} 
 	  The aforementioned references, besides \cite{weaklybounded},  are restricted to the particular case in which $p=\frac{1}{2}$.
 	Since $T$ is compact and $\dim\mathcal{R}(T)=\infty$, it follows that $\mathcal{R}(T)$ is non-closed, which implies that
 	the generalised inverse (see, e.g., \cite{nashed})  $T^\dagger$ is an unbounded operator. 
 	We therefore introduce regularisation. We opt   to employ \textit{Tikhonov regularisation} (cf.~\cite{tikhon}) in which the regularised  solution is given by 
 	\begin{equation*}
 			x^\delta_\alpha:= (T^\ast T+\alpha I)^{-1}T^\ast y^\delta.
 	\end{equation*}
 	We also denote $x_\alpha$ as the regularised solution with exact data.  
 	Note that by \eqref{eq:etadef} and $p\leq \frac{1}{2}$, $x^\delta_\alpha$ is well-defined.
 	Furthermore, we shall assume henceforth that $y\in\mathcal{D}(T^\dagger)$. Then, in the case that $y$ is non-attainable, i.e. $y\notin\mathcal{R}(T)$, we may reduce to the attainable case by considering $Tx=Qy$ where $Q:Y\to\overline{\mathcal{R}(T)}$ is an orthogonal projection (cf.~\cite{engl}). 
 	 	
 	 Our central aim is to approximate the \textit{best approximate solution} $x^\dagger=T^\dagger y$, such that $x^\delta_{\alpha_\ast}$ converges to  $x^\dagger$ in the weakly bounded noise case, 
i.e.,  	
 	as $\eta\to 0$ for an appropriately selected $\alpha_\ast$. 
 	
 	 	In the current setting, (cf.~\cite{egger,weaklybounded,mathegeneralnoise}) the balancing principle or modified discrepancy rules were suggested for the parameter choice. Note that these are a-posteriori rules which require knowledge of the noise level. In practical situations, such information is not normally available and this motivates the need for so-called {\em heuristic parameter choice rules} in which the parameter is selected as the minimiser of a functional $\psi:(0,\|T\|^2)\times Y\to[0,\infty]$, i.e.,
 	\[
 	\alpha_\ast:=\argmin_{\alpha\in(0,\|T\|^2)}\psi(\alpha,y^\delta),
 	\]
 	which requires no knowledge of $\eta$. The main objective of this paper is the analysis of heuristic parameter choice rules 
 	in the weakly bounded noise (aka large noise) case.
 	
 	The functionals $\psi$ in this article may also be represented in terms of spectral theory:
 	\[
 	\psi^2(\alpha,y^\delta)=\int_0^{\|T\|^2} \pa(\lambda)\,\mathrm{d}\|F_\lambda y^\delta\|^2,
 	\]
 	where $\pa:(0,\|T\|^2)\to\mathbb{R}_+$ is a spectral filter function and $\{F_\lambda\}_\lambda$ denotes the spectral family of $TT^\ast$. For later reference we also define $\{E_\lambda\}_\lambda$ to be the spectral family of $T^\ast T$.
 	
 	 Note that in the following, $C$ will denote an arbitrary positive constant which need not be universally equal.
 
    The paper is organised as follows: in the proceeding section, we study and extend the
    classical heuristic parameter choice rules, namely, the quasi-optimality, heuristic discrepancy, and Hanke-Raus rules.  We establish convergence rates  under noise conditions 
    similar to the strongly bounded noise case \cite{kinderabstract,kinderquasi}. 
    In Section~\ref{sec:three}, we investigate  known statistical rules 
    in a deterministic framework, in particular, the generalised cross-validation rule. 
    Since this is only defined in a discrete setting, we first analyse an infinite-dimensional 
    variant, i.e., the predictive mean-square error method.

 	\section{Heuristic Parameter Choice Rules}\label{sec:two}
	The standard method of approach to prove convergence rates for heuristic parameter choice rules is to estimate the data error from above by $\psi(\alpha,y-y^\delta)$ for which we also attain an estimate from above.  One also estimates $\psi(\alpha,y)$ from above.  If $\alpha_\ast$ is the minimiser of $\psi(\alpha,y^\delta)$, then
	\begin{align*}
	\|x^\delta_{\alpha_\ast}-x^\dagger\|&\le\|x^\delta_{\alpha_\ast}-x_{\alpha_\ast}\|+\|x_{\alpha_\ast}-x^\dagger\|=\mathcal{O}(\psi(\alpha_\ast,y-y^\delta)+\|x_{\alpha}-x^\dagger\|)
	\\
	&=\mathcal{O}(\psi(\alpha,y^\delta)+\psi(\alpha_\ast,y)+\|x_{\alpha}-x^\dagger\|),
	\end{align*}
	from which the derivation of the rates is quite standard.

 	Specifically, in this paper, we consider heuristic rules based on the following $\psi$-functionals:
 	\begin{itemize}
 		\item 
 		The quasi-optimality functional (cf.~\cite{tikhon})
 		\begin{equation}\label{defQO}
 		\psi_{\text{QO}}(\alpha,y^\delta):=\alpha\left\|\frac{\mathrm{d}}{\mathrm{d}\alpha}x^\delta_\alpha\right\|,
 		\end{equation}
 		with
 		\[
 		\pp{\alpha,\text{QO}}(\lambda)=\frac{\alpha^2\lambda}{(\lambda+\alpha)^4}.
 		\]
 		\item
 		The modified heuristic discrepancy functional (cf.~\cite{HankeRaus})
 		\[
 		\psi_{\text{HD}}(\alpha,y^\delta):=\frac{1}{\alpha^{q+\frac{1}{2}}}\|(TT^\ast)^q(Tx^\delta_\alpha-y^\delta)\|, \qquad \text{where} \ q\ge p,
 		\]
 		with
 		\[
 		\pp{\alpha,\text{HD}}(\lambda)=\frac{\lambda^{2q}\alpha}{\alpha^{2q}(\lambda+\alpha)^2}.
 		\]
 		\item
 		The modified Hanke-Raus functional (cf.~\cite{HankeRaus})
 		\[
 		\psi_{\text{HR}}(\alpha,y^\delta):=\frac{1}{\alpha^{q+\frac{1}{2}}}\left\langle (TT^\ast)^{q}(Tx^\delta_{\alpha,2}-y^\delta),(TT^\ast)^{q}
 		(Tx^\delta_\alpha-y^\delta)\right\rangle^{\frac{1}{2}},
 		\text{where} \ q\ge p.
 		\]
 		where $x^\delta_{\alpha,2}:=(T^\ast T+\alpha I)^{-1}(T^\ast y^\delta+\alpha x^\delta_\alpha)$ is the second iterated Tikhonov solution, with
 		\[
 		\pp{\alpha,\text{HR}}(\lambda)=\frac{\lambda^{2q}\alpha^{2}}{\alpha^{2q}(\lambda+\alpha)^3}.
 		\]
 	\end{itemize}
Note that our definitions of the heuristic discrepancy and Hanke-Raus functionals 
are generalisations of   the usual ones. The usual functionals are obtained for 
the special case $q = 0$. The reason for this modification is that in 
the setting of weakly bounded noise, the discrepancy is possibly unbounded, and 
hence, the standard functionals need not be bounded either. 
Therefore, by introducing the operator $(TT^\ast)^q$, the functionals become finite if $q$ is chosen larger than $p$.
 This is a simple exercise to prove. Note that the quasi-optimality functional does not 
 require any modification.

The drawback of heuristic parameter choice rules comes in the form of the so-called Bakushinskii veto, which states that choosing the parameter heuristically cannot lead to a convergent regularisation method in the worst case (cf.~\cite{veto}). In spite of this, heuristic rules are still very often used with great success in practice. Motivated by this, it was shown that if one does not consider the worst case, heuristic rules may lead to convergent regularisation methods. In particular, in \cite{kinderquasi,kinderabstract}, additional noise conditions were postulated in order to estimate the data error as
 	 \begin{equation}
 	 \|x^\delta_\alpha-x_\alpha\|\le C\psi(\alpha,y-y^\delta), \label{abstractnoisecondition}
 	 \end{equation} 
 	 from which we can prove convergence of the method. As we will show 
 	 in the subsequent sections (and as was proven for the bounded noise case in \cite{kinderquasi,kinderabstract}), the estimate \eqref{abstractnoisecondition} is obtained
 	  for the mentioned rules whenever we impose a noise condition $y-y^\delta\in\mathcal{N}_\nu$, i.e.,
 	 \begin{equation}\label{defNnu}
 	 \mathcal{N}_\nu:=\left\{e\in Y:\alpha^{\nu+1}\int_{\alpha}^{\|T\|^2}\lambda^{-1}\,\mathrm{d}\|F_\lambda e\|^2\le C\int_0^\alpha\lambda^\nu\,\mathrm{d}\|F_\lambda e\|^2 \right\},
 	 \end{equation}
 	 where $\nu=1$ for $\psi=\psi_{\text{QO}}$ and $\nu=2q$ for $\psi\in\{\psi_{\text{HD}},\psi_{\text{HR}}\}$. 

Let us state some simple examples, where a noise condition \eqref{defNnu} holds, 
and, in particular, convince the reader that the assumption of weakly bounded 
noise is compatible with condition \eqref{defNnu}.  Note that in the classical 
situation of (strongly) bounded noise, it has been verified that \eqref{defNnu} is 
satisfied in typical situations \cite{kinderquasi}. Moreover, for coloured 
Gaussian noise, \eqref{defNnu} holds almost surely \cite{KiPePi18}. 

Suppose that $T T^\ast$ has eigenvalues $\{\lambda_i\}$ with polynomial decay,  and 
we assume a certain polynomial decay or growth of the noise $e = y^\delta-y$ with respect to  the eigenfunctions of $T T^\ast$, denoted by $\{u_i\}$:
\begin{equation}\label{examp} \lambda_i = \frac{1}{i^\gamma},  \quad \gamma > 0, 
\qquad \text{ and } 
\qquad |\langle y^\delta-y,u_i\rangle|^2 = \tau \frac{1}{i^\beta} \qquad \beta \in \mathbb{R}, \tau >0. \end{equation}
Then 
\[ \|y^\delta-y\|^2 = \tau \sum_{i=1}^\infty \frac{1}{i^\beta}, 
\qquad \eta^2 = \|(T T^\ast)^py^\delta-y\|^2  = 
 \tau \sum_{i=1}^\infty \frac{1}{i^{\beta+ 2 p \gamma}}. \]
 If we consider the case of unbounded, but weakly bounded noise, 
 i.e., $\|y^\delta-y\|^2 = \infty$ but $\eta < \infty$, 
 the exponents $\beta, p$ should thus satisfy 

\[  \beta \leq 1 \quad \text{ and } \quad \beta +2 p \gamma >1,
\qquad \text{ thus, } \quad  \beta \in (1-2p \gamma,1]. \] 
The inequality  in \eqref{defNnu} can then be written as 
\[ 
\tau \alpha^{\nu+1} \sum_{1\leq i \leq \alpha^{-\frac{1}{\gamma}}} i^{\gamma - \beta} = 
\alpha^{\nu+1} \sum_{\lambda_i \geq \alpha } \frac{1}{\lambda_i} 
  |\langle y^\delta-y,u_i\rangle|^2 \leq 
 C  \sum_{\lambda_i \leq \alpha } \lambda_i^\nu   |\langle y^\delta-y,u_i\rangle|^2 
 = C \tau  \sum_{i\geq \alpha^{-\frac{1}{\gamma}}} \frac{1}{i^{\gamma \nu + \beta}}. \]
Defining $N_*=\alpha^{-\frac{1}{\gamma}}$, we have
\[
\sum_{1\leq i \leq \alpha^{-\frac{1}{\gamma}}} i^{\gamma - \beta} \leq 
\int_1^{N_*} {x^{\gamma -\beta}} \, \mathrm{d}x 
\leq C  \begin{cases} N_*^{\gamma -\beta +1} &\mbox{if } \gamma -\beta >-1, \\
1 &\mbox{if } \gamma -\beta <-1, \end{cases} 
\] 
and 
\[  \sum_{i\geq \alpha^{-\frac{1}{\gamma}}} \frac{1}{i^{\gamma \nu + \beta}} 
\sim \int_{N^*}^\infty  \frac{1}{x^{\gamma \nu +\beta}} \, \mathrm{d}x \sim
 \begin{dcases} 
 \frac{C}{N_*^{\gamma \nu + \beta -1}}  & \text{ if } \gamma \nu + \beta >1,\\
\infty & \text{ if } \gamma \nu + \beta \leq 1. \end{dcases}
\] 
Since $\alpha= N_*^{-\gamma}$, we arrive at the sufficient inequality 
\[ N_*^{-\gamma(\nu+1) + 1 + \gamma -\beta } \leq C 
N^{1 -\gamma \nu -\beta}, \] 
in the case that $\gamma-\beta >-1$ and $\gamma \nu +\beta > 1$.
Since the exponents 
match, the noise condition is then satisfied. 
If
$\gamma \nu +\beta \leq 1$, then the inequality is clearly satisfied because of 
the divergent right-hand side. Thus, the noise condition holds for 
\[ \beta <\gamma+1. \] 
Roughly speaking,  this means that the noise should not be too regular 
(relative to the smoothing of the operator). In particular, the deterministic 
model of white noise, where $\beta = 0$ (no decay) satisfies a noise condition 
if  the operator is smoothing.  
Most importantly, the assumption of a noise condition 
\eqref{defNnu}  is compatible with a weakly bounded noise situation.

 	 In the latter sections, we also consider the predictive mean-square error (PMS) functional
 	 \cite{Wahba,Lukas93}
 	 \[
 	 \psi_{\text{PMS}}(\alpha,y^\delta):=\|Tx^\delta_\alpha-y\|. 
 	 \]
 	 This is not an implementable  parameter choice rule per se as it involves 
 	 the (unknown) exact data $y$. The reason for opting 
 	 to study this functional is its relation to the generalised cross-validation functional,
 	 which is one of our main aims. For ill-conditioned problems $T_n x=y_n$, where $T_n:X\to\mathbb{R}^n$, the generalised cross-validation (GCV) functional \cite{Wahba} is given by 
 	 \[
 	 \psi_{\text{GCV}}(\alpha,y_n^\delta):=\frac{1}{\rho(\alpha)}\|T_nx^\delta_\alpha-y_n^\delta\|,
 	 \]
 	 with $\rho(\alpha):=\frac{\alpha}{n}\operatorname{tr}\left\{(T_nT^\ast_n+\alpha I_n)^{-1}\right\}$. 
 	 Its relation to  $\psi_{\text{PMS}}$ is that for i.i.d.~noise, 
 	 the expected value of 	 $\psi_{\text{GCV}}^2(\alpha,y_n^\delta) - \|e\|^2$ 
 	 is an estimator for the expected value of the 
 	 predictive mean-square error functional squared, as has been shown by Wahba \cite{Wahba}. 
	For numerical treatment of the GCV method, see, e.g., \cite{Lothar}.

 	\paragraph{Convergence analysis} The convergence analysis of regularisation methods with standard (non-heuristic) parameter choice rules 
 	in the weakly bounded noise setting is well established: 
 	for instance, in the present setting, one can easily show, as in \cite{egger}, that
 	\[
 		\|x^\delta_\alpha-x^\dagger\|\to 0,
 	\]
 	if one chooses $\alpha_\ast$ such that $\alpha_\ast\to 0$ and $\eta^{2+p}/\alpha_\ast\to 0$ as $\eta\to 0$.  Therefore, even in the presence of large noise, one may obtain a convergent regularisation method.
 	
 	We are also interested in deriving rates of convergence.  To this end, we assume throughout that the best approximate solution $x^\dagger\in X$ satisfies the source condition: 
 	\begin{equation}
 	x^\dagger\in\mathcal{R}((T^\ast T)^\mu)
\Longleftrightarrow  	 	x^\dagger = (T^\ast T)^\mu \omega, \quad \|\omega\| < \infty, \qquad 
 	\qquad  0 \leq \mu \leq 1,
 	\label{sourcecond}
 	\end{equation}
 	which one can think of as a kind of \textit{smoothness} condition on the solution.

 	The following error estimates are courtesy of \cite{weaklybounded} (cf. also \cite{mathegeneralnoise,egger}):
 	\begin{proposition}
 	Let $x^\dagger$ satisfy \eqref{sourcecond}. Then
 	\begin{alignat}{2}
 	\|x^\delta_\alpha-x_\alpha\|&\le C\frac{\eta}{\alpha^{p+\frac{1}{2}}},&
 	\qquad
 	\|x_\alpha-x^\dagger\|&\le C\alpha^\mu, \quad \mu \leq 1, \label{ratex}
 	\\
 	\|T(x^\delta_\alpha-x_\alpha)\|&\le C\frac{\eta}{\alpha^p},&
 	\qquad
 	\|Tx_\alpha-y\|&\le C\alpha^{\mu+\frac{1}{2}}, \quad \mu \leq \frac{1}{2},   \label{rateT}
 	\end{alignat}
 	for all $\alpha\in(0,\|T\|^2)$.
\end{proposition}
This proposition also illustrates the fact that convergence rates for Tikhonov regularisation
do not improve for $\mu \geq 1$, which is the well-known saturation effect (cf.~\cite{engl}). 
This is also the reason why we do not assume a source condition in \eqref{sourcecond} with 
$\mu >1$.

	We now consider an a-priori parameter choice  yielding a so-called optimal (order) rate. Thereafter, we will utilise this a-priori parameter choice strategy to deduce convergence rates with respect to the heuristic parameter choice rules. In particular, if $x^\dagger$ satisfies the source condition \eqref{sourcecond}, then using the estimates of the previous proposition, one can estimate the total error as
 		\begin{equation}
 			\|x^\delta_\alpha-x^\dagger\|\le C\alpha^\mu+C\frac{\eta}{\alpha^{p+\frac{1}{2}}}=\mathcal{O}(\eta^{\frac{2\mu}{2\mu+2p+1}}), 
 			\label{totaltriangleerror}
 		\end{equation}
 		 which follows by taking the infimum over all $\alpha$. In particular, one obtains that
 		 \[
 		 \alpha_{\text{opt}}\sim\eta^{\frac{2}{2\mu+2p+1}},
 		 \]
 		 is the so-called optimal (order) parameter choice.

For the following analysis, we state a standard estimate for spectral filter functions: for $ t \geq 0$, there is a constant $C$ such that for all nonegative $\alpha,\lambda$
\begin{equation}\label{genin}
\frac{\lambda^t}{(\alpha +\lambda)} \leq C
\begin{dcases} \frac{1}{\alpha^{1-t}}   & 0\leq t \leq 1 \\ 
     1& t \geq 1 \end{dcases} = \frac{C}{\alpha^{\max\{1-t,0\}}}.    
\end{equation}

\subsection{The Quasi-Optimality Rule}
Following \cite{kinderquasi}, we show some upper and lower bounds for the 
quasi-optimality functional, which subsequently allow us to derive convergence rates. 
  	\begin{proposition}\label{PropQOUP}
 		For all $\alpha\in(0,\|T\|^2)$, one has
 		\begin{align*}
 			\psi_{\text{QO}}(\alpha,y-y^\delta)&\le\|x^\delta_\alpha-x_\alpha\|, \\
 			  	\psi_{\text{QO}}(\alpha,y)&\le\|x_\alpha-x^\dagger\|.
 		\end{align*}
 	\end{proposition}
 	We omit the proof, but one may find it in \cite{kinderquasi}.
 
 Next, by assuming a noise condition, we verify the essential lower bound 
for the data-error part of the quasi-optimality functional: 
 \begin{lemma}\label{lemQOlow}
 		Let  $y-y^\delta\in\mathcal{N}_{1}$. Then there exists a positive constant $C$ such that 
 		\begin{equation*}
 			\psi_{\text{QO}}(\alpha,y-y^\delta)\ge C \|x^\delta_\alpha-x_\alpha\|,
 		\end{equation*}
 		for all $\alpha\in(0,\|T\|^2)$.
 	\end{lemma}
 	
 	\begin{proof}
 		For all $\alpha\in(0,\|T\|^2)$, we can estimate
 			\begin{align*}
 			&\psi_{\text{QO}}^2(\alpha,y-y^\delta)=\int^{\|T\|^2}_0 \frac{\lambda\alpha^2}{(\lambda+\alpha)^4} \,\mathrm{d}\|F_\lambda (y-y^\delta)\|^2 				
 				\ge C\int_0^\alpha\frac{\lambda}{(\lambda+\alpha)^2}\,\mathrm{d}\|F_\lambda (y-y^\delta)\|^2 \\				
 				&\ge C\int_0^\alpha\frac{\lambda}{\alpha^2}\,\mathrm{d}\|F_\lambda (y-y^\delta)\|^2 \ge \frac{1}{\alpha^{2}}C\int_0^\alpha\lambda\,\mathrm{d}\|F_\lambda (y-y^\delta)\|^2
 				\ge \alpha^{2}\frac{1}{\alpha^{2}}\int_\alpha^{\|T\|^2}\lambda^{-1}\,\mathrm{d}\|F_\lambda (y-y^\delta)\|^2,
 			\end{align*}
 		where we introduced the noise condition \eqref{defNnu} for $\nu = 1$ in the penultimate inequality.
From this, 
\[ \frac{\lambda}{(\lambda+\alpha)^2} \leq C \begin{dcases}   \frac{1}{\lambda} &\mbox{if } \lambda \geq \alpha, \\ 
  \frac{\alpha^2\lambda}{(\lambda+\alpha)^4} &\mbox{if } \lambda \leq \alpha, \end{dcases} \] 
and 
\begin{align*}  	\|x^\delta_\alpha-x_\alpha\|^2 =  
 \int_0^{\alpha} \frac{\lambda}{(\lambda+\alpha)^2}\,\mathrm{d}\|F_\lambda (y-y^\delta)\|^2  	+
 				\int_{\alpha}^{\|T\|^2} \frac{\lambda}{(\lambda+\alpha)^2}\,\mathrm{d}\|F_\lambda (y-y^\delta)\|^2,  			
 				\end{align*} 
the result follows.  				
 	\end{proof}
Thus, with the above estimates, we may now state the convergence rate of the total error 
with respect to the regularised solution and the parameter chosen according to the quasi-optimality rule:
 	\begin{theorem}\label{thQO}
 		Let $y-y^\delta\in \mathcal{N}_{1}$, $T^\ast y\ne 0$,
 		$x^\dagger$ satisfy \eqref{sourcecond},  and let $\alpha_\ast$ be the 
 		parameter selected according to  the quasi-optimality  rule. 
 		Then we obtain
 		\begin{equation*}
 			\|x^\delta_{\alpha_\ast}-x^\dagger\|=\mathcal{O}(\eta^{\frac{2\mu}{2\mu+2p+1}\mu}),
 		\end{equation*}
for $\eta$ sufficiently small. 		
 	\end{theorem}
 	\begin{proof}
From  Proposition~\ref{PropQOUP}, the definition of $\alpha_\ast$ and the triangle inequality, it   follows, with 	$\alpha=\eta^{\frac{2}{2\mu+2p+1}}$, that 
	\begin{align*}
 			&\psi^2_{\text{QO}}(\alpha_\ast,y^\delta)\le	\psi^2_{\text{QO}}(\alpha,y^\delta) \le 
\left(\psi_{\text{QO}}(\alpha,y^\delta-y) + \psi_{\text{QO}}(\alpha,y)\right)^2   		\\	
&\leq 
2 \|x_\alpha-x^\dagger\|^2+2 \|x^\delta_\alpha-x_\alpha\|^2 \leq C\alpha^{2\mu}+C\frac{\eta^2}{\alpha^{2p+1}}
= \mathcal{O}\left(\left[\eta^\frac{2\mu}{2\mu+2p+1}\right]^2\right).
 				\end{align*} 	 	
By the triangle inequality and Lemma~\ref{lemQOlow}, 		
 				\begin{align*}
 				\|x^\delta_{\alpha_\ast}-x^\dagger\|&\le\|x_{\alpha_\ast}-x^\dagger\|+\|x_{\alpha_\ast}-x^\delta_{\alpha_\ast}\|
 		=\mathcal{O}\left(\|x_{\alpha_\ast}-x^\dagger\|+\psi_{\text{QO}}(\alpha_\ast,y-y^\delta)\right) \\ 		
 		& \leq 
\mathcal{O}\left( \|x_{\alpha_\ast}-x^\dagger\|+\psi_{\text{QO}}(\alpha_\ast,y^\delta ) + \psi_{\text{QO}}(\alpha_\ast,y) \right)
				=\mathcal{O}\left( \alpha^\mu_\ast+\left[\eta^\frac{2\mu}{2\mu+2p+1}\right]^2 \right).
 		\end{align*}
	Note that 
 		\begin{equation}\label{qolowlow}
 		\begin{split}
 			&\psi_{\text{QO}}^2(\alpha,y^\delta)\ge\alpha^2\int_0^{\|T\|^2} \frac{\lambda}{(\lambda+\|T\|^2)^4}\,\mathrm{d}\|F_\lambda y^\delta\|^2 \geq 
 			\frac{1}{(2\|T\|^2)^4} \int_0^{\|T\|^2}\lambda\, \mathrm{d}\|F_\lambda y^\delta\|^2\\
 			&\geq 
 				\alpha^2 \frac{1}{(2\|T\|^2)^4} \left( \|T^* y\| - \|T T^\ast\|^{\frac{1}{2}-p} \eta\right)^2 \geq 
 				C	\alpha^2,
 			\end{split}	
 		\end{equation}
 		for all $\alpha\in(0,\|T\|^2)$	and $\eta$ sufficiently small. 
  		Hence for $\alpha=\alpha_\ast$,  it follows that $\alpha_\ast \leq  C \eta^\frac{2\mu}{2\mu+2p+1}$. 		
 		Therefore, we may deduce that
 		\begin{equation*}
 			\|x^\delta_{\alpha_\ast}-x^\dagger\|=\mathcal{O}(\eta^{\frac{2\mu}{2\mu+2p+1}\mu}),
 		\end{equation*}
 		for $\eta$ sufficiently small, which is what we wanted to show.
 	\end{proof}
	One may notice that the above convergence rates are  optimal for the 
	saturation case $\mu = 1$, but they are only suboptimal for $\mu <1$ (similarly as in \cite{kinderabstract}).
		 We may, however, impose an additional condition in order to achieve an optimal
	convergence rate. More specifically, we impose the following \textit{regularity condition} on the best approximate solution, $x^\dagger\in X$:
 	\begin{equation}
 		\alpha^2\int_\alpha^\infty \lambda^{-2}\,\mathrm{d}\|E_\lambda x^\dagger\|^2\ge C\int_0^\alpha\mathrm{d}\|E_\lambda x^\dagger\|^2. \label{sourcecondition2}
 	\end{equation}
This condition was also used in \cite{kinderquasi,kinderabstract} where it was 
shown that it is often satisfied.  
 	\begin{theorem}
 	 In addition to the assumptions of Theorem~\ref{thQO}, let the regularity condition 
 	 \eqref{sourcecondition2} hold. 
 Then there exists a constant $C>0$ such that
 		\begin{equation*}
 			\|x^\delta_{\alpha_\ast}-x^\dagger\|\le C\inf_{\alpha\in(0,\|T\|^2)}\left(\|x_\alpha-x^\dagger\|+\|x_\alpha-x^\delta_\alpha\|\right),
 		\end{equation*}
 		which yields the optimal convergence rate.
 	\end{theorem}
 	\begin{proof}
 		Recall that $C\|x^\delta_\alpha-x_\alpha\|\le\psi_{\text{QO}}(\alpha,y-y^\delta)\le \|x^\delta_\alpha-x_\alpha\| $ and 
 		the regularity condition \eqref{sourcecondition2} imply cf.~\cite{kinderabstract} that 
 		 $C\|x_\alpha-x^\dagger\| \leq 		\psi_{\text{QO}}(\alpha,y)\le C \|x_\alpha-x^\dagger\|$, 
 		 from which the theorem follows similar to \cite{kinderabstract}. 	
 	\end{proof} 	
In essence, the stated theorems are completely analogous to the bounded noise case \cite{kinderabstract}.  	
Thus, all the known results for the quasi-optimality principle extend to the 
weakly-bounded noise case. For advanced numerical implementations of this method, 
see \cite{RausHam}.

	\subsection{The Modified Heuristic Discrepancy Rule} 	
Now we prove convergence rates for the modified heuristic discrepancy rule 
in a similar way by proving estimates for the functional acting on the noise and exact data.
Note that this method
is sometimes confusingly  also referred to as the Hanke-Raus rule (as both rules 
agree for Landweber iteration). For clarity,  it is preferable to name this method as the heuristic analogue of the classical discrepancy rule.

The upper bounds for the functional are straightforward to derive:
 	\begin{proposition}
 		For  $p\leq q \leq 1 + p$, we have  
 		\begin{equation}
 			\psi_{\text{HD}}(\alpha,y^\delta-y)\le C \frac{\eta}{\alpha^{\frac{1}{2}+p,}} \label{hdupperboundnoise}
 		\end{equation}
 		for all $\alpha\in(0,\|T\|^2)$.
 		Let $x^\dagger$ satisfy \eqref{sourcecond} and suppose $q  \le\frac{1}{2}-\mu$. Then there exists a positive constant $C$ such that
 		\begin{equation}
 			\psi_{\text{HD}}(\alpha,y)\le C\alpha^\mu. \label{hdexactupperbound}
 		\end{equation}
 	\end{proposition}
 	\begin{proof}
 	This follows easily from the spectral representation and the inequalities 
 	\begin{align*} &\frac{ \lambda^{2q}}{\alpha^{1 +2q}}\frac{\alpha^2}{(\lambda+\alpha)^2} \leq  
 	 \frac{ 1}{\alpha^{2q-1}}\frac{\lambda^{2(q-q)}}{(\lambda+\alpha)^2} \lambda^{2p} \leq   	
 	C \frac{1}{\alpha^{2q -1 + 2\max\{1-(q-p),0\}}}\lambda^{2p}  \\
 	& = C   \frac{1}{\alpha^{\max\{2  p +1,2 q-1\}}} \lambda^{2p}.
 	\end{align*} 
 	The second result follows from 
 	 	\begin{align*} &\frac{\lambda^{2q}}{\alpha^{1 +2q}} \frac{\alpha^2}{(\lambda+\alpha)^2} \lambda^{1 + 2 \mu} 
 	 	\leq   C \frac{1}{\alpha^{2q-1}} \frac{1}{\alpha^{2\max\{1 - \frac{1}{2} - \mu - q,0\}}} = 
 	 C \frac{1}{\alpha^{\max\{ -2 \mu , 2q-1\}}} = 	C \alpha^{\min\{2 \mu, 1- 2 q\}}.
 	 	\end{align*}
 	\end{proof}
 		\begin{proposition}
 	Let $ q  \leq \frac{1}{2}$. If $Q(y-y^\delta)\in\mathcal{N}_{2 q}$, then
 		\begin{equation}
 			\psi_{\text{HD}}(\alpha,y-y^\delta)\ge\|x^\delta_\alpha-x_\alpha\|, \label{hdnoiselowerbound}
 		\end{equation}
 		for all $\alpha\in(0,\|T\|^2)$.
 	\end{proposition}
 	\begin{proof}
 	We estimate
 			\begin{equation}
 				\begin{aligned}
 					\psi^2_{\text{HD}}(\alpha,y-y^\delta)&=\frac{1}{\alpha^{2q+1}}\int_0^{\|T\|^2}\lambda^{2q}\frac{\alpha^2}{(\lambda+\alpha)^2}\,\mathrm{d}\|F_\lambda Q(y-y^\delta)\|^2
 					\\
 					&\ge 
 					C \frac{1}{\alpha^{2q+1}}
 					\int_0^\alpha\lambda^{2q}\,\mathrm{d}\|F_\lambda Q(y-y^\delta)\|^2+
 					C 	\frac{1}{\alpha^{2q-1}}	\int_\alpha^{\|T\|^2}\lambda^{2q-2}
 					\mathrm{d}\|F_\lambda Q(y-y^\delta)\|^2, \label{hdnoisebelow}
 				\end{aligned}
 			\end{equation}
 			for all $\alpha\in(0,\|T\|^2)$.
 			
 			Conversely, 
 			\begin{equation} \label{eq:dataupper}
 			\begin{split}
  				&\|x_{\alpha}^\delta - x_\alpha\|^2 =
 				\int_0^\alpha\frac{\lambda}{(\lambda+\alpha)^2}\,\mathrm{d}\|F_\lambda Q(y-y^\delta)\|^2 \\
 				&\leq 
 				C\frac{1}{\alpha}	
 				\int_0^\alpha 	\frac{\lambda}{\alpha}\,
 				\mathrm{d}\|F_\lambda Q(y-y^\delta)\|^2
 				+ C
 				\int_\alpha^{\|T\|^2} \lambda^{-1}	\,
 				\mathrm{d}\|F_\lambda Q(y-y^\delta)\|^2.
 			\end{split}	
 			\end{equation}	
 			Since $2q-1 \leq 0$, we observe that the term 
 			with $\int_0^\alpha$ in the above inequality
 			is bounded by the corresponding term in 
 			\eqref{hdnoisebelow}. 
 			Thus, using the noise condition, the 
 			second term can be bounded by the first one
 			of 	\eqref{hdnoisebelow}.
 	\end{proof}
\begin{theorem}\label{th5}
Let $p\leq q \leq p +1$, $q\le\frac{1}{2}-\mu$ and suppose the noise condition $Q(y-y^\delta)\in\mathcal{N}_{2 q}$. 
Moreover, suppose that $(T T^\ast)^q Q y \not = 0$ and let $x^\dagger$ satisfy \eqref{sourcecond}.
 		Let $\alpha_\ast$ be selected according to the modified heuristic discrepancy rule. 
 		Then for $\eta$ sufficiently small, 
 		\begin{equation*}
 		\|x^\delta_{\alpha_\ast}-x^\dagger\|=\mathcal{O}\left(\eta^{\frac{2\mu}{2\mu+2p+1}\frac{2\mu}{1-2q}}\right).
 		\end{equation*}
 	\end{theorem}
 	\begin{proof}
 		Note that from $(T T^\ast)^q Q y \not = 0$, we may conclude, as in \eqref{qolowlow}, that 
 		\[
 		\alpha_\ast \leq C \left(\psi_{\text{HD}}(\alpha,y^\delta)^{\frac{1}{\frac{1}{2}-q}}\right)
 		 = \mathcal{O} \left(\eta^{\frac{2\mu}{2\mu+2p+1}\frac{2}{1-2q}} \right).
 		\]
 		Then it follows, as above,  from \eqref{hdnoiselowerbound}, \eqref{hdupperboundnoise}, 
 		and \eqref{hdexactupperbound},  that 
 		\begin{align*}
 		&\|x^\delta_{\alpha_\ast}-x^\dagger\|\le\|x_{\alpha_\ast}-x^\dagger\|+\|x_{\alpha_\ast}-x^\delta_{\alpha_\ast}\| 	
 		=\mathcal{O}(\alpha^\mu_\ast+\psi(\alpha_\ast,y-y^\delta)) =  
 		\\
 		&=\mathcal{O}\left(\alpha_\ast^{\mu}+\alpha^{\mu}+\frac{\eta}{\alpha^{\frac{1}{2}+p}}\right)
 		=\mathcal{O}\left(\eta^{\frac{2\mu}{2\mu+2p+1}\frac{2\mu}{1-2q}}+\eta^{\frac{2\mu}{2\mu+2p+1}}\right),
 		\end{align*} 		
 		for $\eta$ sufficiently small. This proves the theorem. 
 	\end{proof}	 
Let us discuss the assumptions in this theorem: the first condition on $q$ is not particularly 
restrictive. However, the requirement $q\le\frac{1}{2}-\mu$ implies that 
$\mu \leq \frac{1}{2}- q$, which means that we obtain a  staturation at 
$\mu=\frac{1}{2}- q$. This is akin to the bounded noise case ($q =0$), where this 
method saturates at $\mu = \frac{1}{2}$. It is well known that a similar phenomenon occurs for the  non-heuristic analogue of this method, namely the discrepancy principle.

Similarly as for the quasi-optimality method, we again only obtain suboptimal rates except 
for the saturation case $\mu = \frac{1}{2}- q$. However, using again a regularity condition, 
we can even prove optimal order  convergence rates for the modified heuristic discrepancy rule. 
\begin{theorem}
 	Let $p\leq q \leq p +1$ and $q\le\frac{1}{2}-\mu$. Assume the noise condition
 	$Q(y-y^\delta)\in\mathcal{N}_{2 q}$, source condition  \eqref{sourcecond}, 
 	and regularity condition \eqref{sourcecondition2} hold. 
 		Furthermore, let $\alpha_\ast$ be selected according to the modified heuristic discrepancy rule. Then
 		\begin{equation*}
 		\|x^\delta_{\alpha_\ast}-x^\dagger\|=\mathcal{O}\left(\eta^{\frac{2\mu}{2\mu+2p+1}}\right).
 		\end{equation*}
 	\end{theorem}
 \begin{proof} 
 We show that a regularity condition implies that 
$\psi_{\text{HD}}(\alpha,y) \geq C \|x_\alpha -x^\dagger\|$.
Recall that 
\begin{align}\label{sss} \|x_\alpha -x^\dagger\|^2 = 
\int_0^{\|T\|^2} \frac{\alpha^2}{(\alpha +\lambda)^2} \, \mathrm{d} \|E_\lambda x^\dagger\|^2 
\leq 
C \int_0^{\alpha} \, \mathrm{d} \|E_\lambda x^\dagger\|^2 + 
C \alpha^2 \int_{\alpha}^{\|T\|^2} \frac{1}{\lambda^2}\, \mathrm{d} \|E_\lambda x^\dagger\|^2. 
\end{align} 
On the other hand, 
\begin{align*} &\psi^2_{\text{HD}}(\alpha,y) \geq 
\int_0^{\|T\|^2}  \left(\frac{\lambda}{\alpha}\right)^{2 q}  
\frac{\lambda \alpha }{(\lambda + \alpha)^2} \,
 \mathrm{d} \|E_\lambda x^\dagger\|^2  \geq 
\int_\alpha^{\|T\|^2} \left(\frac{\lambda}{\alpha}\right)^{2q} 
\frac{\lambda \alpha }{(\lambda + \alpha)^2} \,
 \mathrm{d} \|E_\lambda x^\dagger\|^2 \\
 &\geq 
  C \int_\alpha^{\|T\|^2} 
 \frac{\alpha}{\lambda} \,
 \mathrm{d} \|E_\lambda x^\dagger\|^2 \geq 
C \int_\alpha^{\|T\|^2} 
 \frac{   \alpha^2}{\lambda^2} \,
 \mathrm{d} \|E_\lambda x^\dagger\|^2. 
 \end{align*}
  By the regularity condition, the first integral in the upper bound in 
 \eqref{sss} can be estimated by the second part which agrees up to a constant with 
 the lower bound for $\psi^2_{\text{HD}}(\alpha,y)$.  In the proof 
 of Theorem~\ref{th5}, the estimate $\|x_{\alpha_\ast} - x^\dagger\| \leq C \alpha_{\ast}^\mu$
 can then be replaced by $\|x_{\alpha_\ast} - x^\dagger\| \leq C \psi(\alpha_\ast,y),$ 
 which leads, as in \cite{kinderquasi}, to the optimal rate.
\end{proof}

\subsection{The Modified Hanke-Raus Rule} 	
 As for the other parameter choice rules, we prove estimates for the modified Hanke-Raus functional with the aim of deriving convergence rates.
	\begin{proposition}
 		Let $p\leq q \leq p + \frac{3}{2}$. Then there exists a positive constant $C$ such that
 		\[
 		\psi_{\text{HR}}(\alpha,y-y^\delta)\le C\frac{\eta}{\alpha^{p+\frac{1}{2}}},
 		\]
 		for all $\alpha\in(0,\|T\|^2)$.
 		Let $q\le 1  -\mu$. Then there exists a positive constant $C$ such that
 		\[
 		\psi_{\text{HR}}(\alpha,y)\le C\alpha^{\mu},
 		\]
 		for all $\alpha\in(0,\|T\|^2)$.
 	\end{proposition}
 	\begin{proof}
In terms of filter functions, we have 
\begin{align*}
\frac{\lambda^{2q}\alpha^{2-2q}}{(\lambda+\alpha)^3}  = 
\frac{\lambda^{2(q-p)}\alpha^{2-2q}}{(\lambda+\alpha)^3} \lambda^{2p} \leq 
C  \frac{\alpha^{2-2q}}{\alpha^{\max\{3 - 2(q-p),0\}}}\lambda^{2p} \leq 
C  \frac{1}{\alpha^{\max\{1 +2p,2q-2\}}}\lambda^{2p},
\end{align*} 
which leads to the first estimate. 
The second statement follows from 
\begin{align*} \frac{\lambda^{2q}\alpha^{2-2q}}{(\lambda+\alpha)^3} \lambda^{1+2\mu} 
\leq C  \frac{\alpha^{2-2q}}{\alpha^{\max\{3 - (1 +2 \mu + 2q),0\}  }} 
\leq C  \frac{1}{\alpha^{\max\{-2 \mu,2q-2\}  }}. 
\end{align*}
 	\end{proof}
The lower bound for the noise propagation again  requires a noise condition. 
 		\begin{proposition}
 		Let $Q(y-y^\delta)\in\mathcal{N}_{2q}$. Then there exists a positive constant $C$ such that
 		\[\psi_{\text{HR}}(\alpha,y-y^\delta)\ge C\|x^\delta_\alpha-x_\alpha\|, \]
 		for all $\alpha\in(0,\|T\|^2)$.
 	\end{proposition}
 	\begin{proof}
We estimate   	
\[ \frac{ \lambda^{2q}}{\alpha^{2q}}\frac{\alpha^2}{(\lambda+\alpha)^3} \geq 
C \begin{dcases} \frac{\lambda^{2q}}{\alpha^{2q+1}} &\mbox{if } \lambda \leq \alpha, \\ 
\frac{\lambda^{2q-3}}{\alpha^{2q-2}}
&\mbox{if } \lambda \geq \alpha. \end{dcases} 
\] 
Now, using $\mathcal{N}_{2q}$ and \eqref{eq:dataupper}, we can estimate $\|x^\delta_\alpha-x_\alpha\|$ by the part of
 $\psi_{\text{HR}}(\alpha,y-y^\delta)$ restricted to $\lambda \leq \alpha$. The part for $\lambda \geq \alpha$ can then be estimated from below by $0$. 
\end{proof}
	\begin{theorem}\label{MHRth}
	Let $p\leq q \leq p + \frac{3}{2}$ and $q \leq 1 -\mu$. 
 		Moreover, suppose  $Q(y-y^\delta)\in\mathcal{N}_{2 q}$, 
 		$ (TT^\ast)^q Q y \not = 0$ and let $x^\dagger$ satisfy
 		\eqref{sourcecond}. Let $\alpha_\ast$ be chosen 
 		according to the modified Hanke-Raus rule. 
 		Then
 		\[
 		\|x^\delta_{\alpha_\ast}-x^\dagger\|=\mathcal{O}\left(\eta^{\frac{2\mu}{2\mu+2p+1}\frac{\mu}{1-q}} \right).
 		\]
 	\end{theorem}
 	\begin{proof}
 	As in \eqref{qolowlow}, 
 		 we prove that if  $\|(TT^\ast)^q Q y\|\ge C$, then
		\[
 		\alpha_\ast\le C\psi_{\text{HR}}(\alpha_\ast,y^\delta)^{\frac{1}{1-q}}.
 		\] 		
 		Therefore,
 		\begin{align*}
 		&\|x^\delta_{\alpha_\ast}-x^\dagger\|\le\|x^\delta_{\alpha_\ast}-x_{\alpha_\ast}\|+\|x_{\alpha_\ast}+x^\dagger\|
 		=\mathcal{O}\left(\alpha_\ast^\mu+\psi_{\text{HR}}(\alpha,y)+\psi_{\text{HR}}(\alpha,y-y^\delta) \right)
 		\\
 		&=\mathcal{O}\left(\psi_{\text{HR}}(\alpha,y^\delta)^{\frac{\mu}{1-q}}+\eta^{\frac{2\mu}{2\mu+2p+1}} \right)
 		=\mathcal{O}\left(\eta^{\frac{2\mu}{2\mu+2p+1}\frac{\mu}{1-q}} \right).
 		\end{align*}

 	\end{proof}
In contrast to the modified discrepancy rule, we observe that the saturation occurs 
at $\mu = 1 - q$. Hence, again analogous to the bounded noise case (and to the 
non-heuristic case), the modified Hanke-Raus method yields convergence rates for a wider 
range of smoothness classes.  The observed suboptimal rates for the 
non-saturation case can again be handled by a regularity condition. 

\begin{theorem}
Suppose that, in addition to the assumptions in Theorem~\ref{MHRth}, the 
 regularity condition \eqref{sourcecondition2} holds.  
 Then, 
 		\begin{equation*}
 		\|x^\delta_{\alpha_\ast}-x^\dagger\|=\mathcal{O}\left(\eta^{\frac{2\mu}{2\mu+2p+1}}\right).
 		\end{equation*}
 	\end{theorem}
 \begin{proof}
Similarly as for the modified discrepancy rule, we estimate  
\begin{align*} &\psi^2_{\text{HR}}(\alpha,y) \geq 
\int_0^{\|T\|^2}  \frac{\lambda^{2q+1} \alpha^{2-2q}}{(\lambda + \alpha)^3} 
 \, \mathrm{d} \|E_\lambda x^\dagger\|^2  \geq 
 \alpha^2 \int_\alpha^{\|T\|^2} \left(\frac{\lambda}{\alpha}\right)^{2q} 
 \frac{\lambda} {(\lambda + \alpha)^3} 
 \, \mathrm{d} \|E_\lambda x^\dagger\|^2 \\
 &\geq 
  \alpha^2 \int_\alpha^{\|T\|^2} 
 \frac{\lambda}{\lambda^3} 
 \, \mathrm{d} \|E_\lambda x^\dagger\|^2 \geq 
   \alpha^2 \int_\alpha^{\|T\|^2} 
 \frac{1}{\lambda^2} 
 \, \mathrm{d} \|E_\lambda x^\dagger\|^2. 
 \end{align*}
As before, combined with the regularity condition, this 
allows to conclude that  $\psi_{\text{HR}}(\alpha,y) \geq C \|x_\alpha - x^\dagger\|$,
and the rest of the proof follows similarly as for the modified heuristic discrepancy case. 
\end{proof} 

\section{PMS and GCV}\label{sec:three}
In this section we study the generalised cross-validation and its 
infinite-dimensional analogue, the  predictive mean-square error method, 
in a deterministic framework.

\subsection{The Predictive Mean-Square Error}
The predictive mean-square error functional differs from the 
previous ones in the sense that it has different upper bounds. 
In fact, from \eqref{rateT}, one immediately finds that 
\[ 		\psi_{\text{PMS}}^2(\alpha,y^\delta) \leq 
C \frac{\eta^2}{\alpha^{2p}} + C\alpha^{2 \mu+1}, \] 
for $\mu \leq \frac{1}{2}$.  
The minimum of the upper bound is again obtained for $\alpha = \alpha_{\text{opt}}
= \mathcal{O}(\eta^{\frac{2}{2 p + 2 \mu+1}})$, 
but the resulting rate is of the order 
\[ 		\psi_{\text{PMS}}^2(\alpha,y^\delta) \leq C \left[\eta^{\frac{(2 \mu+1)}{2 p + 2 \mu+1}}\right]^2,
\]
which agrees with the optimal rate for the error in the $T$-norm, 
$\|x^\delta_\alpha-x^\dagger\|_T:=\|T(x_\alpha^\delta - x^\dagger)\|$. Thus, for this method, it is not reasonable to 
bound the functional $\psi_{\text{PMS}}$ by expressions involving 
$\|x_{\alpha}^\delta - x_{\alpha}\|$ or $\|x_{\alpha} - x^\dagger\|$. 
Rather, we try to directly relate  the selected regularisation parameter $\alpha_\ast$
to the optimal choice $\alpha_{\text{opt}}$.

To do so, we need some estimates from below, 
	although in this case, we will need to introduce a noise condition of a different type and an additional condition on the exact solution.
 	\begin{lemma}\label{lemma2}
 		Suppose that there exists a positive constant $C$ such that $y^\delta-y\in Y$ satisfies
 		\begin{equation}\label{regcon}  
 		 \int_{\alpha}^{\|T\|^2} \mathrm{d}\|F_\lambda Q(y-y^\delta)\|^2 \geq C  \frac{\eta^2}{\alpha^{2p-\varepsilon}}, 
 		\end{equation}
%
 		for all $\alpha\in(0,\|T\|^2)$ and $\varepsilon>0$ small. 		
 		Then
 		\[
 		\begin{aligned}
 		\|T(x^\delta_\alpha-x_\alpha)\|&\ge C\frac{\eta}{\alpha^{p-\frac{\varepsilon}{2}}}.
  		\end{aligned}
  		\]
 	\end{lemma}
\begin{proof}
From \eqref{regcon}, one can estimate
\begin{align*}
	\|T(x^\delta_\alpha-x_\alpha)\|^2 
	= \int_0^{\|T\|^2} \frac{\lambda^2}{(\alpha+\lambda)^2} \, 
	 \mathrm{d}\|F_\lambda Q(y-y^\delta)\|^2 \geq  
		\int_\alpha^{\|T\|^2}   \, 
	 \mathrm{d}\|F_\lambda Q(y-y^\delta)\|^2 \geq   
C  \frac{\eta^2}{\alpha^{2p-\varepsilon}}. 	
\end{align*}
\end{proof}

Let us exemplify condition \eqref{regcon}: for the case in \eqref{examp}, we have that 
\[ 
	 \int_{\lambda\geq\alpha} \mathrm{d}\|F_\lambda Q(y-y^\delta)\|^2 
	 = \sum_{1\leq i \leq N_*} \frac{1}{i^\beta} \sim 
	 \int_1^{N_*}  \frac{1}{x^\beta}\, \mathrm{d}x = 
	 \begin{cases} C N_*^{1-\beta} &\mbox{if } 1-\beta >0, \\ 
	               C &\mbox{if } 1-\beta <0, \end{cases} 
\]
with $N_* = \frac{1}{\alpha^\frac{1}{\gamma}}$. This gives that the left-hand side 
is of the order of ${\alpha^{-\frac{1-\beta}{\gamma}}}$. For \eqref{regcon} to hold 
true, we require that $\frac{1-\beta}{\gamma} \geq {2p-\varepsilon}$, which means that 
\[ 1 + \varepsilon \gamma \geq \beta + 2 p \gamma. \] 
If we now choose $p$ close to the smallest admissible exponent for the 
weakly bounded noise condition, i.e. $2 p \gamma = 1-\beta +\varepsilon \gamma$,
with $\varepsilon$ small, then the condition holds. In other words, our interpretation 
of the stated 
noise condition means that $\|(TT^\ast)^{p}(y^\delta-y)\|< \infty$ and 
$p$ is selected as the minimal exponent such that  this holds. This noise 
condition automatically excludes the (strongly) bounded noise case. It can easily be seen that for strongly bounded noise  $\|y^\delta -y\| <\infty$, the method 
fails as it selects $\alpha_\ast = 0$.
The example also shows that the desired inequality with $\varepsilon =0$
cannot be achieved. 


	\begin{theorem}
	Let $\mu \leq \frac{1}{2}$,  $\alpha_\ast$ be the minimiser of $\psi_{\text{PMS}}(\alpha,y^\delta)$, assume that the noise satisfies \eqref{regcon} and that $T x^\dagger \not = 0$. 
 		Then
 		\[
 		\|x^\delta_{\alpha_\ast}-x^\dagger\|\le
 		\begin{cases}
 		C\eta^{\frac{2\mu}{2\mu+2p+1} \frac{2 \mu+1}{2}},&\mbox{if }\alpha_\ast\ge\alpha_{\text{opt}},
 		\\
 		C\eta^{\frac{2\mu}{2 \mu + 2 p+1} - \epsilon \frac{2p+1}{(2 p-\epsilon)(2 \mu +2 p+1)}},
 		&\mbox{if }\alpha_\ast\le\alpha_{\text{opt}}.
 		\end{cases}
 		\]
If additionally for some $\epss >0$,
\begin{equation}\label{xxx}     \int_{\alpha}^{\|T\|^2} \lambda^{2\mu-1} 	 \,\mathrm{d}\|E_\lambda \omega \|^2 \geq C  \alpha^{2 \mu-1 + \epss}, \end{equation} 
then for the first case we have 
\[ 	\|x^\delta_{\alpha_\ast}-x^\dagger\|\le
 		C\eta^{\frac{2\mu}{2\mu+2p+1} \frac{2 \mu+1}{2 \mu +1 + \epss}}, \qquad \mbox{if }\alpha_\ast\ge\alpha_{\text{opt}}, \]

 	\end{theorem}
 	\begin{proof}
 		If $\alpha_\ast\ge\alpha_{\text{opt}}$, 
 		it follows from $T x^\dagger\not = 0$ that 
 		\[ \|T x_\alpha - y \|^2 \geq C \alpha^2, \] 
 		and 
 		if  \eqref{xxx} holds, then one even has that
 		\[ \|T (x_\alpha - x^\dagger) \|^2 \geq \int_\alpha^{\|T\|^2}  \frac{\lambda^{1 +2 \mu} \alpha^2}{(\alpha +\lambda)^2}
 		 \,\mathrm{d}\|E_\lambda \omega \|^2  \geq \alpha^2 \int_\alpha^{\|T\|^2} \lambda^{2 \mu-1} 	 \,\mathrm{d}\|E_\lambda \omega \|^2
\geq  		 
 		 		C \alpha^{2 \mu+1 +\epss}. \]  		
 		Since $\alpha\mapsto\|T(x^\delta_\alpha-x_\alpha)\|^2$ is a monotonically decreasing function and using Young's inequality, we may obtain that
 		\[ C \alpha_{\ast}^t  \leq
 		\|T(x_{\alpha_{\text{opt}}}^\delta-x_{\alpha_{\text{opt}}})\|^2+ \|Tx_{\alpha_{\text{opt}}}-y\|^2\le C\left[\eta^{\frac{2\mu+1}{2\mu+2p+1}}\right]^2,
 		\]
 		i.e.,
 		\[
 		\alpha_\ast\le C\eta^{\frac{2\mu+1}{2\mu+2p+1} \frac{2}{t}},
 		\]
 		where $t = 2$ or $t = {2 \mu +1 +\epss}$ if \eqref{xxx} holds. 
 		
 		 		If $\alpha_\ast\le\alpha_{\text{opt}}$, then
 		we may bound the functional from below as
 		\[
 		\psi_{\text{PMS}}^2(\alpha,y^\delta)\ge\frac{1}{2}\|T(x^\delta_\alpha-x_\alpha)\|^2-\|Tx_\alpha-y\|^2,
 		\]
 		for all $\alpha\in(0,\|T\|^2)$, which allows us to obtain
 		\begin{align*}
 		&\frac{1}{2}\|T(x^\delta_{\alpha_\ast}-x_{\alpha_\ast})\|^2-
 		\|Tx_{\alpha_\ast}-y\|^2\le\psi_{\text{PMS}}^2(\alpha_\ast,y^\delta)
 		\le\psi_{\text{PMS}}^2(\alpha_{\text{opt}},y^\delta)
 		\\
 		&\le 2\|T(x^\delta_{\alpha_{\text{opt}}}-x_{\alpha_{\text{opt}}})\|^2+2\|Tx_{\alpha_{\text{opt}}}-y\|^2
 		\le C\left[\eta^{\frac{2\mu+1}{2\mu+2p+1}}\right]^2.
 		\end{align*}
 		i.e., by Lemma~\ref{lemma2}, 
 		\[
 		C\frac{\eta^2}{\alpha_\ast^{2p-\varepsilon}}-C\alpha_\ast^{2\mu+1}\le\frac{1}{2}\|T(x^\delta_{\alpha_\ast}-x_{\alpha_\ast})\|^2-\|T_{\alpha_\ast}-y\|^2\le C\left[\eta^{\frac{2\mu+1}{2\mu+2p+1}}\right]^2.
 		\] 	
 		Now, from  $\alpha_\ast\le\alpha_{\text{opt}}$, we get
 		\begin{align*}
 		C\frac{\eta^2}{\alpha_\ast^{2p-\varepsilon}}&\le C\left[\eta^{\frac{2\mu+1}{2\mu+2p+1}}\right]^2+C\alpha_\ast^{2\mu+1}
 		\le C\left[\eta^{\frac{2\mu+1}{2\mu+2p+1}}\right]^2+C\alpha_{\text{opt}}^{2\mu+1}
 		\le C\left[\eta^{\frac{2\mu+1}{2\mu+2p+1}}\right]^2,
 		\end{align*}
 		i.e.,
 		\[
 		\alpha_\ast^{2p-\varepsilon}\geq C\left[\eta^{\frac{2p}{2\mu+2p+1}}\right]^2
 		\Longleftrightarrow 	\alpha_\ast\ge C \eta^{\frac{2}{2\mu+2p+1}\cdot\frac{2p}{2p-\varepsilon}}.
 		\]
 		Then inserting the respective bounds for $\alpha_\ast$ into \eqref{totaltriangleerror} yields the desired rates.
 	\end{proof} 	
Condition \eqref{xxx} can again be verified as we did for the noise condition 
for some canonical examples. The inequality with $\epss=0$ does not usually hold. 
The condition can be interpreted as the claim that $x^\dagger$ satisfies a source condition 
with  a certain  $\mu$ but this exponent cannot be increased, i.e., 
$x^\dagger \not \in \mathcal{R}((T^\ast T)^{\mu +\epsilon})$. A similar condition 
was used by Lukas in his analysis of the generalised cross-validation rule
\cite{Lukas93}. 

The theorem shows that we may obtain almost optimal convergence results 
but only under rather restrictive conditions. Moreover, the method shows 
a saturation effect at $\mu = \frac{1}{2}$ comparable to the discrepancy principles.

\subsection{The Generalised Cross-Validation Rule} 	
The generalised cross-validation rule was proposed and studied in particular 
by Wahba \cite{Wahba}, and it is most popular in a statistical context 
but less so for deterministic inverse problems. It is derived 
from the cross-validation method by combining the associated estimates 
with certain weights. Most importantly, it was shown in \cite{Wahba} 
that the expected value of the generalised cross-validation functional 
converges to the expected value of  the PMS-functional as the 
dimension tends to infinity. This is why, in the last section, we studied 
$\psi_{\text{PMS}}$ in detail. 

One can show that the weight $\rho(\alpha)$ in  $\psi_{\text{GCV}}$ is monotonically 
increasing with $\rho(0)=0$ and bounded with $\rho(\alpha)\leq 1$. It can furthermore 
be shown that for $\alpha >0$, $\rho(\alpha)\to 1$ as the dimension $n \to \infty$.
This is also the reason why one has to study the GCV in terms of weakly bounded noise. 
The limit $\lim_{n\to \infty} \psi_{\text{GCV}}$ tends pointwise to the 
residual $\|T x_{\alpha}^\delta - y_\delta\|$, which in the bounded noise case 
does not yield a reasonable parameter choice as then $\alpha_{\ast} = 0$
is always chosen. 

Note that in a stochastic context, and using the expected value
of  $\psi_{\text{GCV}}$, a convergence analysis has been done by 
Lukas \cite{Lukas93}. In contrast, we analyse the deterministic 
case.

We now consider the ill-conditioned problem
 	\begin{equation}
 		T_n x=y_n,
 	\end{equation}
 	where we only have noisy data $y_n^\delta\in\mathbb{R}^n$.

 We impose a discretisation independent source condition, that is, 
 \[ x^\dagger = (T_n^*T_n)^\mu \omega, \qquad \|\omega\| \leq C, \qquad  0< \mu \leq 1,\]  
 where $C$ does not depend on the dimension $n$.
 Furthermore, let us restate some definitions for this discrete setting: 
\[ \delta_n:= \|y_n^\delta - y_n\|, \qquad \eta^2 := \sum_{i=1}^n
\lambda_i^{2p}  |\langle y_n^\delta - y_n,u_i\rangle|^2.  \] 
Note that in an asymptotically weakly bounded noise case, we might assume that 
$\eta$ is bounded independent of $n$ while $\delta_n$ might be unbounded as 
$n$ tends to infinity.

Moreover, we impose a noise  condition of similar type as for the predictive mean-square error
 \begin{equation}\label{regcond3}
 	 \sum_{\lambda_i \geq \alpha} |\langle y_n^\delta - y,u_i\rangle|^2 \geq C  \frac{\eta^2}{\alpha^{2p-\varepsilon}}, \qquad \text{for all }  \alpha \in \mathrm{I},  
 \end{equation}
  where $C$ does not depend on $n$.
  Note that in the discrete case, one must restrict the noise condition to an interval
 with $\mathrm{I} = [\alpha_{\text{min}},\|T\|^2]$ with $\alpha_{\text{min}} >0$. 
 
Similarly, we state a regularity condition 
 \begin{equation}\label{regcond4}
 	 \sum_{\lambda_i \geq \alpha} \lambda^{2 \mu-1} 
 	 |\langle \omega ,v_i\rangle|^2 \geq C \alpha^{2 \mu-1 +\epss} \qquad \text{for all }  \alpha \in \mathrm{I},
 \end{equation}
 where $\{v_i\}$ denote the eigenfunctions of $T^\ast T$.
 
 	In order to deduce convergence rates, we look to bound the functional from above as we did for the other functionals in the previous sections:
 	\begin{lemma}
 		For $y_n^\delta\in\mathbb{R}^n$, there exist positive constants such that
 		\begin{align}
 		        \psi_{\text{GCV}}(\alpha,y_n)\le \frac{C}{\rho(\alpha)}C \alpha^{2\mu+1},   \qquad \mu \leq \frac{1}{2},  \label{gcvfirst}\\ 
 		        \psi_{\text{GCV}}(\alpha,y_n^\delta-y)\le \frac{C}{\rho(\alpha)} \delta_n^2, \qquad \mu \leq \frac{1}{2}, \qquad \text{ hence,}
 		        \label{gcvsecond}\\ 
 			\psi_{\text{GCV}}(\alpha,y_n^\delta)\le\frac{1}{\rho(\alpha)}\left(C\alpha^{2\mu+1}+\delta_n^2\right),  \qquad \mu \leq \frac{1}{2}. 			
 			\label{dgcvboundabove}
 		\end{align}	
 	\end{lemma}
 	\begin{proof}
 It is a standard result \cite{engl} that  $\|T_n x^\delta_\alpha-y_n^\delta - (T_n x_\alpha -y_n)\| \leq \|y_n^\delta - y_n\| \leq \delta_n$. 
 Similarly, by the usual source condition, we obtain $\|(T_n x_\alpha -y_n)\| \leq C \alpha^{2\mu+1}$ for $\mu \leq \frac{1}{2}$.  
 The result follows from the triangle inequality. 
 	\end{proof}
 	The proceeding results generally follow from the infinite dimensional setting and we similarly obtain the following bounds from below:
 	\begin{lemma}
 	Suppose that $\alpha \in \mathrm{I}$ and also that \eqref{regcond3} holds. 
 	Then 
 	\[ \psi_{\text{GCV}}(\alpha,y_n^\delta-y) \geq  \frac{1}{\rho(\alpha)} \left(C\frac{\eta^2}{\alpha^{2p-\varepsilon}}\right). \] 
        Moreover,  if $\|T_n x^\dagger\|  \geq C_0$, with an $n$-independent constant,  then there exists an
         $n$-independent constant $C$ with 
        \[ \psi_{\text{GCV}}(\alpha,y) \geq C  \frac{1}{\rho(\alpha)}\alpha^2.  \] 
        If  \eqref{regcond4} holds and  $\alpha \in \mathrm{I}$, then 
         \[ \psi_{\text{GCV}}(\alpha,y) \geq C \frac{1}{\rho(\alpha)} \alpha^{2 \mu+1 +\epss}, \qquad \mu \leq
         \frac{1}{2}.  \] 
 	\end{lemma}

 	\begin{theorem}\label{thGCV} Let $\mu \leq \frac{1}{2}$, 
 		assume $\alpha_\ast$ is the minimiser of $\psi_{\text{GCV}}(\alpha,y_n^\delta)$ and
suppose further that $\alpha_\ast \in \mathrm{I}$ such that \eqref{regcond3} holds. 
Then 
\[  
\alpha_\ast\ge \left[\inf_{\alpha \geq \alpha_\ast}
(C\alpha^{2\mu+1}+C\delta^2_n) \right]^{-\frac{1}{2p-\varepsilon}}\eta^{\frac{2}{2p-\varepsilon}} 
\geq C \delta_n^{-\frac{2}{2p-\varepsilon}}\eta^{\frac{2}{2p-\varepsilon}}.  \] 
On the other hand 
	 	\[
	 	\alpha_\ast\le\left[\inf_{\alpha \leq \alpha_\ast} \frac{1}{\rho({\alpha})}  
	 	\left(C\alpha^{2\mu+1}+C\delta^2_n\right) \right]^{\frac{1}{t}},
	 	\]
with $t =2$.  If $\alpha_\ast \in \mathrm{I}$ and \eqref{regcond4} hold, then 
$t = 2 \mu +1 +\epss$.

 	\end{theorem}
 	\begin{proof}
Take an arbitrary  $\bar{\alpha}$ and consider first the case $\alpha_\ast\le\bar{\alpha}$.
 		Following on from the previous lemmas and using \eqref{gcvsecond}, we have 
\begin{align*} 
&\frac{1}{\rho(\alpha_\ast)}\left(C\frac{\eta^2}{\alpha_\ast^{2p-\varepsilon}}\right) \leq 
\psi_{\text{GCV}}^2(\alpha_\ast,y_n^\delta-y) \leq 
C \psi_{\text{GCV}}^2(\alpha_\ast,y_n^\delta)   + C \psi_{\text{GCV}}^2(\alpha_\ast,y_n)   \\ 
& \leq  \psi_{\text{GCV}}^2(\bar{\alpha},y_n^\delta) +  C \frac{1}{\rho(\alpha_\ast)}\alpha_\ast^{2 \mu+1}  \leq 
\frac{1}{\rho(\bar{\alpha})}\left(C\bar{\alpha}^{2\mu+1}+\delta_n^2\right)  +  C\frac{1}{\rho(\alpha_\ast)} \alpha_\ast^{2 \mu+1}. 
\end{align*} 
Hence, by the monotonicity of $\alpha \mapsto \alpha^{2\mu+1}$
and since $\rho$ is monotonically increasing, we obtain that 
\[ \left(C\frac{\eta^2}{\alpha_\ast^{2p-\varepsilon}}\right) \leq  
\frac{\rho(\alpha_\ast)}{\rho(\bar{\alpha})}\left(C\bar{\alpha}^{2\mu+1}+\delta_n^2\right)  +   \alpha_\ast^{2 \mu+1} \leq 
\left(C\bar{\alpha}^{2\mu+1}+\delta_n^2\right)  +   \alpha_\ast^{2 \mu+1} \leq 
\left(C\bar{\alpha}^{2\mu+1}+\delta_n^2\right). 
\]
 Hence,
\[
\alpha_\ast\ge \left[\inf_{\alpha \geq \alpha_\ast}
(C\alpha^{2\mu+1}+C\delta^2_n) \right]^{-\frac{1}{2p-\varepsilon}}\eta^{\frac{2}{2p-\varepsilon}} 
\geq C \delta_n^{-\frac{2}{2p-\varepsilon}}\eta^{\frac{2}{2p-\varepsilon}}. 
\]
Now, suppose $\alpha_\ast\ge\bar{\alpha}$. Then using that $\alpha_\ast$ is a minimiser 
\begin{align*}
&\frac{C}{\rho(\alpha_\ast)}\alpha_\ast^{t}
\leq  \psi_{\text{GCV}}^2(\alpha_\ast,y_n) \leq 
 C\psi_{\text{GCV}}^2(\alpha_\ast,y_n^\delta) +  \psi_{\text{GCV}}^2(\alpha_\ast,y_n^\delta -y_n)   \\
 & \leq \frac{1}{\rho(\bar{\alpha})}(C\bar{\alpha}^{2\mu+1}+C\delta^2_n) + C \frac{1}{\rho({\alpha_\ast})} \delta^2_n \leq 
 \frac{1}{\rho(\bar{\alpha})}(C\bar{\alpha}^{2\mu+1}+C\delta^2_n) + C \frac{1}{\rho({\bar{\alpha}})} \delta^2_n.
\end{align*}
Hence, as $\rho(\alpha_\ast)$ is bounded from above by $1$,  it follows that 
	 	\[
	 	\alpha_\ast\le\left[\inf_{\alpha \leq \alpha_\ast} \frac{1}{\rho({\alpha})}  
	 	\left(C\alpha^{2\mu+1}+C\delta^2_n\right) \right]^{\frac{1}{t}}.
	 	\]
 	\end{proof}
 	
\begin{theorem}
Suppose that, in addition to the assumptions in the previous theorem, one has
 $\rho(\delta_n^{\frac{2}{2 \mu+1}}) \geq C$. Then 
  \[ \|x^\delta_{\alpha_{\ast}} -x^\dagger \| \leq \delta_n^{\frac{2\mu}{t}}  + 
\delta_n\left(\frac{\eta}{\delta_n}\right)^\frac{1}{2p-\varepsilon},
 \] 
with $t$ as in Theorem~\ref{thGCV}.
\end{theorem}  	 	
\begin{proof}
Since 
\[ \|x^\delta_{\alpha_{\ast}} -x^\dagger \| \leq C \alpha^{\mu} + C \frac{\delta_n}{\sqrt{\alpha}},  \]
we may take the balancing parameter $\bar{\alpha} = \delta_n^{\frac{2}{2 \mu+1}}$. 
From the previous theorem, it follows that if $\alpha_\ast \leq \bar{\alpha}$, 
then 
\[ \alpha_\ast \geq 
\frac{\eta^{\frac{2}{2p-\varepsilon}}}{\left[\inf_{\alpha \geq \alpha_\ast}
(C\alpha^{2\mu+1}+C\delta^2_n) \right]^{\frac{1}{2p-\varepsilon}}} 
\geq \left(\frac{\eta}{\delta_n}\right)^\frac{2}{2p-\varepsilon}.
\] 
 On the other hand, if $\alpha_\ast \geq \bar{\alpha}$,  and $\rho(\bar{\alpha}) \geq C$,  then 
 \[  \alpha_\ast  \leq C \delta^{\frac{2}{t}}. \] 
 Thus, taking for $\alpha^{\mu}$ and  $\frac{\delta_n}{\sqrt{\alpha}}$ the worst of these 
 estimates, we obtain the desired result.
\end{proof} 	
This result establishes convergence rates in the discrete case. 
However, the required conditions are somewhat restrictive as we need that 
the selected $\alpha_\ast$ has to be in a certain interval (although this is 
to be expected in a finite-dimensional setting). Note that the term $\delta_n^2$ in 
Theorem~\ref{thGCV} can be replaced by any reasonable monotonically decreasing 
upper bound for $\psi_{\text{GCV}}^2(\alpha,y_\delta -y)$. In particular, if we 
could conclude that $\alpha_\ast$ is in a region where 
$\psi_{\text{GCV}}^2(\alpha,y_\delta -y) \leq C \frac{\eta^2}{\alpha^{2 p}}$,
then we would obtain similar convergence results as for the predictive mean square error. 

In general, however, the performance of the GCV-rule for the regularisation 
of deterministic inverse problems is subpar compared to other heuristic rules, 
e.g., those mentioned in the previous sections; cf., e.g., \cite{HamaPalmRaus,bauerlukas}. 
This is also illustrated by the fact that we had to impose stronger conditions 
for the convergence results compared to the aforementioned  rules.

\section{Conclusion}
We analysed and provided conditions for the derivation of convergence rates for a number of 
well-known heuristic parameter choice rules in the weakly bounded noise setting and modified them when necessary. The theory was extended in a consistent and systematic way whereby one attains the standard results whenever the situation is as in the classical setting. In particular, we provided noise conditions which are very often satisfied for when one can prove suboptimal convergence rates for the quasi-optimality, modified heuristic discrepancy and Hanke-Raus rules, as well as optimal rates whenever certain regularity conditions are satisfied.

A further novel aspect of this paper was the examination of the generalised cross-validation rule and the predictive mean-square error in a deterministic framework. In the case of the former, it was in a finite-dimensional setting where we proved convergence rates. 

In essence, it was demonstrated that heuristic rules remain viable methods for selecting the regularisation parameter, even in the case where the noise is only \emph{weakly bounded}.

 	\bibliographystyle{siam}
 	\bibliography{foo}
 	
 \end{document}